\documentclass[letterpaper,10pt]{amsart}

\usepackage[all]{xy}                        %

\CompileMatrices                            

\UseTips                                    

\input xypic
\usepackage[bookmarks=true]{hyperref}       

\usepackage{amssymb,latexsym,amsmath,amscd}
\usepackage{xspace}
\usepackage{color}
\usepackage{kpfonts}
\usepackage{graphicx}
\usepackage{multicol}
\reversemarginpar

\theoremstyle{plain}
\newtheorem{theorem}{Theorem}[section]
\newtheorem*{theorem*}{Theorem}

\newtheorem{lemma}[theorem]{Lemma}

\theoremstyle{definition}
\newtheorem{definition}[theorem]{Definition}

\newtheorem{remark}[theorem]{Remark}

\newcommand{\enm}[1]{\ensuremath{#1}}          %

\newcommand{\cal}[1]{\mathcal{#1}}

\newcommand{\CC}{\enm{\mathbb{C}}}

\newcommand{\ZZ}{\enm{\mathbb{Z}}}

\newcommand{\PP}{\enm{\mathbb{P}}}

\newcommand{\Aa}{\enm{\cal{A}}}
\newcommand{\Bb}{\enm{\cal{B}}}

\newcommand{\Ee}{\enm{\cal{E}}}
\newcommand{\Ff}{\enm{\cal{F}}}
\newcommand{\Gg}{\enm{\cal{G}}}

\newcommand{\Oo}{\enm{\cal{O}}}

\newcommand{\Vv}{\enm{\cal{V}}}

\renewcommand{\phi}{\varphi}
\renewcommand{\theta}{\vartheta}
\renewcommand{\epsilon}{\varepsilon}

\renewcommand{\to}[1][]{\xrightarrow{\ #1\ }}

\newcommand{\old}[1]{}

\newcommand{\ppp}{{\mathbb{P}^1 \times \mathbb{P}^1 \times \mathbb{P}^1 }}

\begin{document}

\title[On Primitive Ulrich Bundles over a few projective varieties with Picard number two]{On Primitive Ulrich Bundles over a few projective varieties with Picard number two}

\author{F. Malaspina }

\address{Politecnico di Torino, Corso Duca degli Abruzzi 24, 10129 Torino, Italy}
\email{francesco.malaspina@polito.it}



\begin{abstract}
We introduce the notion of primitive Ulrich bundle in a smooth projective variety. We motivate this notion and give a cohomological characterization in the case of the degree $6$ flag threefold  and rational normal scrolls. Finally we propose a few open problems.
\end{abstract}

\maketitle
{\large Dedicated to the memory of Gianfranco Casnati}

\section{Introduction}
A locally free sheaf (or ``bundle'') $\Ee$ on a projective varity $X$ is ACM if it has no intermediate cohomology or if the module
$E$ of global sections of
$\Ee$ is a maximal Cohen-Macaulay module.
There has been increasing interest on the classification of  ACM bundles on various projective varieties, which is important in a sense that the ACM bundles are considered to give a measurement of complexity of the underlying space.  A special type of ACM sheaves, called the Ulrich sheaves, are the ones achieving the maximum possible minimal number of generators. These bundles
are characterized by the linearity of the minimal graded free resolution
over the polynomial ring of their module of global section.
Ulrich bundles, originally studied for computing Chow forms,
conjecturally exist over any variety (see \cite{ES}). 

 In this article we introduce the notion of primitive Ulrich  bundle as an Ulrich bundle which is extension of direct sums of Ulrich line bundles.
If we consider the varieties with a finite number of ACM bundles (see \cite{EH}) the projective spaces $\PP^n$, the hyperquadrics $\mathcal Q_n$, the Veronese surface $V_2$ and the cubic scroll $S(1,2)$ we notice that, except for the cases of $\mathcal Q_n$ with $n>2$, most of the Ulrich bundles (actually also of the ACM bundles) are primitive.\\
In \cite{FM} it has been showed that the quartic scroll surfaces $S(1,3)$ and $S(2,2)$ support at most one dimensional families of Ulrich bundles. An explicit classification it is given and we can notice that all the one dimensional families are made by primitive Ulrich bundles. Also on elliptic curves there are at most one dimensional families of 
(primitive) Ulrich bundles (see \cite{At}). On Segre varieties $\PP^{n_1}\times\dots\times\PP^{n_s}$, with $n_1\leq\dots\leq n_s$, it is possible to find arbitrary large dimensional  families of Ulrich bundles; when $n_1=1$,   arbitrary large families of primitive Ulrich bundles has been constructed in \cite{CMP}. See \cite{M} for Segre-Veronese varieties, \cite{ACM} for ruled surfaces and \cite{A} for Hirzebruch surfaces. Also on different type of threefold scrolls  arbitrary large families of primitive Ulrich bundles has been constructed (see \cite{FLP} and \cite{FF}). We may conclude that primitive bundles play an important role within Ulrich bundles and that they are worth investigating. Here we  give  cohomological characterizations of primitive Ulrich bundles on the degree $6$ flag threefold  and rational normal scrolls and we propose a few open problems

Here we summarize the structure of this article. In section \ref{sec2} we introduce the definition of primitive Ulrich bundles and several notions in derived category of coherent sheaves to understand the Beilinson spectral sequence. In section \ref{sec3} we deal with the case of the degree $6$ flag threefold. In section \ref{sec4} we study the cases of  rational normal scrolls of arbitrary dimension.
In section \ref{sec5} we discuss a few open problems.


\section{Preliminaries}\label{sec2}
Throughout the article our base field is the field of complex numbers $\CC$. We denote by $X$ a smooth projective variety over $\CC$ with a fixed ample line bundle $\Oo_X(1)$.

\begin{definition}
A coherent sheaf $\Ee$ on a projective variety $X$  is called {\it arithmetically Cohen-Macaulay} (for short, ACM) if it is locally Cohen-Macaulay and $H^i(\Ee(t))=0$ for all $t\in \ZZ$ and $i=1, \ldots, \dim (X)-1$.
\end{definition}



\begin{definition}
For an {\it initialized} coherent sheaf $\Ee$ on $X$, i.e. $h^0(\Ee(-1))=0$ but $h^0(\Ee)\ne 0$, we say that $\Ee$ is an {\it Ulrich sheaf} if it is ACM and $h^0(\Ee)=\deg (X)\mathrm{rank}(\Ee)$.
\end{definition}
\begin{remark} The following conditions are equivalent (see \cite{ES}):

\begin{itemize}
\item[$(i)$] $E$ is Ulrich.
\item[$(ii)$] $E$ admits a linear $\Oo_{\mathbb P^N}$ -resolution of the form:
$$0\to \Oo_{\mathbb P^N} (-N+n)^{a_{N-n}}\to\dots\to \Oo_{\mathbb P^N} (-1)^{a_{1}} \to\Oo_{\mathbb P^N}^{a_{0}}\to E \to 0. $$ 
where $a_0 = rank(E) deg(X)$ and $$a_i=\left( \begin{array}{c} N-n \\ i \end{array} \right)
a_0$$ for all $i.$

\item[$(iii)$] $H^i(E(-i))=0$  for $i>0$ and $H^i(E(-i-1))=0$ for $i< n.$
\end{itemize}
Moreover, since $X$ smooth, an Ulrich sheaf is always locally free.
\end{remark}
\begin{definition}
A vector bundle $E$ over $X$ is said primitive Ulrich  bundle if it is an Ulrich bundle which is extension of direct sums of Ulrich line bundles. So $E$ is a primitive Ulrich bundles if there exist $A=\oplus_{i=1}^s L_i$ and $B=\oplus_{j=1}^z L'_j$,  with $L_i,L'_j$  Ulrich line bundles, such that $E$ arises from the following exact sequence
$$0\to A\to E\to B\to 0.$$
In particular $s$ or $z$ be can be $0$ so also Ulrich line bundles can be consider as primitive Ulrich bundles.
\end{definition}
Let $\dim(X)\geq 2$. Eisenbud and Herzog (see \cite{EH}) classified the varieties with a finite number of ACM bundles: the projective spaces $\PP^n$, the hyperquadrics $\mathcal Q_n$, the Veronese surface $V_2$ i.e. $(\PP^2,\Oo_{\PP^2}(2))$ and the cubic scroll $S=S(1,2)$ i.e. the hyperplane section of $X=\mathbb P^1\times \mathbb P^2$. In the following table in the second column we list the ACM bundles up to twists, no twists of which are not Ulrich and in the third column the Ulrich bundles:
\begin{center}
\begin{tabular}{|l|c|r|}
\hline
&ACM & ULRICH\\
\hline
$\PP^n$ & &$\Oo$  \\
\hline
$\mathcal Q_n$ &$\Oo$ & $\Sigma_i$\\
\hline
$V_2$& $\Oo, \Oo(1)$ & $\Omega_{\mathbb P^2}(3)$\\
\hline
$S(1,2)$ & $\Oo, \Oo(1,0)$ & $\Oo(2,0), \Oo(0,1), E$\\
\hline
\end{tabular}
\end{center}
 where $\Sigma_*$ are the spinor bundles (we use an unified notation  for spinor bundles on $\mathcal Q_n$, where for even $n$, $i$ can take on the values $1,2$, while if $n$ is odd, i can be only $1$) and $E$ arises from the unique ($h^1(\Oo_S(-2,1))=1$) extension
  
$$0\to\Oo_S(0,1)\to E\to\Oo_S(2,0)\to 0.$$
We notice that, except for the cases of $\mathcal Q_n$ with $n>2$, most of the Ulrich bundles (actually also of the ACM bundles) are primitive.\\
Even in the other varieties with "few" ACM bundles, as was observed in the introduction, the larger families are made up of primitive Ulrich bundles.\\

 In the next section we will give some cohomological characterization of primitive Ulrich bundles. An useful tool will be Beilinson spectral sequences:

Given a smooth projective variety $X$, let $D^b(X)$ be the the bounded derived category of coherent sheaves on $X$. An object $E \in D^b(X)$ is called {\it exceptional} if $Ext^\bullet(E,E) = \mathbb C$.
A set of exceptional objects $\langle E_0, \ldots, E_n\rangle$ is called an {\it exceptional collection} if $Ext^\bullet(E_i,E_j) = 0$ for $i > j$. An exceptional collection is said to be {\it full} when $Ext^\bullet(E_i,A) = 0$ for all $i$ implies $A = 0$, or equivalently when $Ext^\bullet(A, E_i) = 0$ does the same.

\begin{definition}\label{def:mutation}
Let $E$ be an exceptional object in $D^b(X)$.
Then there are functors $\mathbb L_{E}$ and $\mathbb R_{E}$ fitting in distinguished triangles
$$
\mathbb L_{E}(T) 		\to	 Ext^\bullet(E,T) \otimes E 	\to	 T 		 \to	 \mathbb L_{E}(T)[1]
$$
$$
\mathbb R_{E}(T)[-1]	 \to 	 T 		 \to	 Ext^\bullet(T,E)^* \otimes E	 \to	 \mathbb R_{E}(T)	
$$
The functors $\mathbb L_{E}$ and $\mathbb R_{E}$ are called respectively the \emph{left} and \emph{right mutation functor}.
\end{definition}


The collections given by
\begin{align*}
E_i^{\vee} &= \mathbb L_{E_0} \mathbb L_{E_1} \dots \mathbb L_{E_{n-i-1}} E_{n-i};\\
^\vee E_i &= \mathbb R_{E_n} \mathbb R_{E_{n-1}} \dots \mathbb R_{E_{n-i+1}} E_{n-i},
\end{align*}
are again full and exceptional and are called the \emph{right} and \emph{left dual} collections. The dual collections are characterized by the following property; see \cite[Section 2.6]{GO}.
\begin{equation}\label{eq:dual characterization}
Ext^k(^\vee E_i, E_j) = Ext^k(E_i, E_j^\vee) = \left\{
\begin{array}{cc}
\mathbb C & \textrm{\quad if $i+j = n$ and $i = k $} \\
0 & \textrm{\quad otherwise}
\end{array}
\right.
\end{equation}

\begin{theorem}[Beilinson spectral sequence]\label{thm:Beilinson}
Let $X$ be a smooth projective variety and with a full exceptional collection $\langle E_0, \ldots, E_n\rangle$ of objects for $D^b(X)$. Then for any $A$ in $D^b(X)$ there is a spectral sequence
with the $E_1$-term
\[
E_1^{p,q} =\bigoplus_{r+s=q} Ext^{n+r}(E_{n-p}, A) \otimes \mathcal H^s(E_p^\vee )
\]
which is functorial in $A$ and converges to $\mathcal H^{p+q}(A)$.
\end{theorem}

The statement and proof of Theorem \ref{thm:Beilinson} can be found both in  \cite[Corollary 3.3.2]{RU}, in \cite[Section 2.7.3]{GO} and in \cite[Theorem 2.1.14]{BO}.


Let us assume next that the full exceptional collection  $\langle E_0, \ldots, E_n\rangle$ contains only pure objects of type $E_i=\mathcal E_i^*[-k_i]$ with $\mathcal E_i$ a vector bundle for each $i$, and moreover the right dual collection $\langle E_0^\vee, \ldots, E_n^\vee\rangle$ consists of coherent sheaves. Then the Beilinson spectral sequence is much simpler since
\[
E_1^{p,q}=Ext^{n+q}(E_{n-p}, A) \otimes E_p^\vee=H^{n+q+k_{n-p}}(\mathcal E_{n-p}\otimes A)\otimes E_p^\vee.
\]

Note however that the grading in this spectral sequence applied for the projective space is slightly different from the grading of the usual Beilison spectral sequence, due to the existence of shifts by $n$ in the index $p,q$. Indeed, the $E_1$-terms of the usual spectral sequence are $H^q(A(p))\otimes \Omega^{-p}(-p)$ which are zero for positive $p$. To restore the order, one needs to change slightly the gradings of the spectral sequence from Theorem \ref{thm:Beilinson}. If we replace, in the expression
\[
E_1^{u,v} = \mathrm{Ext}^{v}(E_{-u},A) \otimes E_{n+u}^\vee=
H^{v+k_{-u}}(\mathcal E_{-u}\otimes A) \otimes \mathcal F_{-u}
\]
$u=-n+p$ and $v=n+q$ so that the fourth quadrant is mapped to the second quadrant, we obtain the following version (see \cite{AHMP}) of the Beilinson spectral sequence:



\begin{theorem}\label{use}
Let $X$ be a smooth projective variety with a full exceptional collection
$\langle E_0, \ldots, E_n\rangle$
where $E_i=\mathcal E_i^*[-k_i]$ with each $\mathcal E_i$ a vector bundle and $(k_0, \ldots, k_n)\in \ZZ^{\oplus n+1}$ such that there exists a sequence $\langle F_n=\mathcal F_n, \ldots, F_0=\mathcal F_0\rangle$ of vector bundles satisfying
\begin{equation}\label{order}
\mathrm{Ext}^k(E_i,F_j)=H^{k+k_i}( \mathcal E_i\otimes \mathcal F_j) =  \left\{
\begin{array}{cc}
\mathbb C & \textrm{\quad if $i=j=k$} \\
0 & \textrm{\quad otherwise}
\end{array}
\right.
\end{equation}
i.e. the collection $\langle F_n, \ldots, F_0\rangle$ labelled in the reverse order is the right dual collection of $\langle E_0, \ldots, E_n\rangle$.
Then for any coherent sheaf $A$ on $X$ there is a spectral sequence in the square $-n\leq p\leq 0$, $0\leq q\leq n$  with the $E_1$-term
\[
E_1^{p,q} = \mathrm{Ext}^{q}(E_{-p},A) \otimes F_{-p}=
H^{q+k_{-p}}(\mathcal E_{-p}\otimes A) \otimes \mathcal F_{-p}
\]
which is functorial in $A$ and converges to
\begin{equation}
E_{\infty}^{p,q}= \left\{
\begin{array}{cc}
A & \textrm{\quad if $p+q=0$} \\
0 & \textrm{\quad otherwise.}
\end{array}
\right.
\end{equation}
\end{theorem}

\section{The flag variety $F(0,1,2)$}\label{sec3}
In this section we give a cohomological characterization of primitive Ulrich bundles over the flag variety $F(0,1,2)$
Let $F\subseteq\mathbb P^7$ be the del Pezzo threefold  of degree $6$ and Picard number two. Let us call $h_1, h_2$ the generators of the Picard group. Let us consider $F$ as an hyperplane section of $\mathbb P^2\times\mathbb P^2$ with the two natural projections $p_i:F\subset\mathbb P^2\times\mathbb P^2\to\mathbb P^2$ and the following rank two vector bundles:
$$
\Gg_1=p_1^*\Omega_{\mathbb P^2}^1(h_1)=p_1^*[\Omega_{\mathbb P^2}^1(1)]\qquad \Gg_2=p_2^*\Omega_{\mathbb P^2}^1(h_2)=p_2^*[\Omega_{\mathbb P^2}^1(1)],
$$
We write $\Oo_F(a,b)$ instead of $\Oo_F(ah_1+bh_2)$. We have the exact sequences
\begin{equation}\label{a1}0 \to \Oo_F(-2,0) \to
 \Oo_F^{3}(-1,0)\to\Gg_1\to 0
 \end{equation}

 \begin{equation}\label{a2}0 \to \Oo_F(0,-2) \to
 \Oo_F^{3}(0,-1)\to\Gg_2\to 0
 \end{equation}
 
 \begin{equation}\label{a3}0 \to\Gg_1\to \Oo_F^3 \to
 \Oo_F(1,0)\to 0
 \end{equation}
 
 \begin{equation}\label{a4}0 \to\Gg_2\to \Oo_F^3 \to
 \Oo_F(0,1)\to 0
 \end{equation}
 All the rank two ACM bundles has been classified in \cite{CFM}.\\
We may consider the full exceptional collection
\begin{equation}\label{col11} \{E_5=\Oo_F(-1,-1)[-2], E_4 = \Oo_F(-1,0)[-2], E_3 = \Oo_F(0,-1)[-1],\end{equation}
$$E_2 = \Gg_2[-1] , E_1 = \Gg_1, E_0 = \Oo_F\}$$

and
\begin{equation}\label{col22} \{F_0 = \Oo_F, F_1 = \Oo_F(-1,0), F_2 = \Oo_F(0,-1), F_3 =  \Oo_F(0,-2), F_4 = \Oo_F(-2,0), F_5 = \Oo_F(-1,-1)\}
\end{equation}

\begin{theorem}\label{volon44}
Let $\Vv$ be an Ulrich bundle on $F$ such that $h^2(\Vv(-2,-2)\otimes\Gg_1\otimes\Gg_2)=0$.  Then $\Vv$ is primitive and arises from an exact sequence of the form:
\begin{equation}\label{res22}
0\to\Oo_F(0,2)^{\oplus a}\to\Vv\to\Oo_F(2,0)^{\oplus b}\to 0.\end{equation}

\end{theorem}
\begin{proof}
We consider the Beilinson type spectral sequence associated to $\Aa:=\Vv(-2,-2)$ and identify the members of the graded sheaf associated to the induced filtration as the sheaves mentioned in the statement. We  consider the full exceptional collection $\Ee_{\bullet}$ and right dual collection $\Ff_{\bullet}$ in (\ref{col11}) and (\ref{col22}).

 We construct a Beilinson complex, quasi-isomorphic to $\Aa$, by calculating $H^{i+k_j}(\Aa\otimes \Ee_j)\otimes \Ff_j$ with  $i,j \in \{0, \ldots, 6\}$ to get the following table:

 \begin{center}\begin{tabular}{|c|c|c|c|c|c|c|c|c|c|c|}
\hline
 $\Oo_F(-1,-1)$ & $\Oo_F(-2,0)$ & $\Oo_F(0,-2)$ & $\Oo_F(0,-1)$ & $\Oo_F(-1,0)$ & $\Oo_F$ \\
 \hline
 \hline
$H^3$	&	$H^3$	&	$*$		&	$*$		&	$*$		&	$*$	\\
$H^2$	&	$H^2$	&	$H^3$	&	$H^3$	&	$*$		&	$*$	\\
$H^1$	& 	$H^1$	&	$H^2$	&	$H^2$	&	$H^3$	&	$H^3$	\\
$H^0$	& 	$H^0$	&	$H^1$	&	$H^1$	&	$H^2$	&	$H^2$	\\
$*$		&	$*$ 	 	&	$H^0$	&	$H^0$	& 	$H^1$	& 	$H^1$	\\
$*$		&	$*$		&	$*$		&	$*$		&	$H^0$	& 	$H^0$ \\
\hline
$\Oo_F(-1,-1)$		& $\Oo_F(-1,0)$	 & $\Oo_F(0,-1)$	 & $\Gg_2$		& $\Gg_1$		&$\Oo_F$\\
\hline
\end{tabular}
\end{center}

We assume due to \cite[Proposition 2.1]{ES} that
$$H^i(\Aa(-j,-j))=0 \text{ for all $i$ and $-1\le j \le 1$}.$$
From the exact sequence
$$0\to\Aa\otimes\Gg_1\to\Aa^{\oplus 3}\to\Aa(1,0)\to 0,$$
since,  by (\ref{a1}), $H^3(\Aa\otimes\Gg_1)=0$ we get $H^2(\Aa(1,0))=0$ and since  $H^0(\Aa(1,0))=0$ we get $H^1(\Aa\otimes\Gg_1)=0$. In a similar way we get $H^2(\Aa(0,1))=H^1(\Aa\otimes\Gg_2)=0$. 
If we twist the above sequence by $\Oo_F(0,1)$ we get also $H^2(\Aa\otimes\Gg_1(0,1))=0$ and an analog way $H^2(\Aa\otimes\Gg_2(1,0))=0$.\\
From the exact sequence
$$0\to\Aa(-1,0)\to\Aa^{\oplus 3}\to\Aa\otimes\Gg_1(1,0)\to 0,$$
since, by (\ref{a3}), $H^0(\Aa\otimes\Gg_1(1,0))=0$ we get $H^1(\Aa(-1,0))=0$ and since $H^3(\Aa(-1,0))=0$ we get $H^2(\Aa\otimes\Gg_1(1,0))=0$.  In a similar way we get $H^1(\Aa(0,-1))=H^2(\Aa\otimes\Gg_2(0,1))=0$.

So the table become
 \begin{center}\begin{tabular}{|c|c|c|c|c|c|c|c|c|c|c|}
\hline
 $\Oo_F(-1,-1)$ & $\Oo_F(-2,0)$ & $\Oo_F(0,-2)$ & $\Oo_F(0,-1)$ & $\Oo_F(-1,0)$ & $\Oo_F$ \\
 \hline
 \hline
$0$	&	$0$	&	$*$		&	$*$		&	$*$		&	$*$	\\
$0$	&	$H^2$	&	$0$	&	$0$	&	$*$		&	$*$	\\
$0$	& 	$0$	&	$H^2$	&	$H^2$	&	$0$	&	$0$	\\
$0$	& 	$0$	&	$0$	&	$0$	&	$H^2$	&	$0$	\\
$*$		&	$*$ 	 	&	$0$	&	$0$	& 	$0$	& 	$0$	\\
$*$		&	$*$		&	$*$		&	$*$		&	$0$	& 	$0$ \\
\hline
$\Oo_F(-1,-1)$		& $\Oo_F(-1,0)$	 & $\Oo_F(0,-1)$	 & $\Gg_2$		& $\Gg_1$		&$\Oo_F$\\
\hline
\end{tabular}
\end{center}
From the exact sequence
$$0\to\Aa\otimes\Gg_1\otimes\Gg_2\to\Aa^{\oplus 3}\otimes\Gg_2\to\Aa(1,0)\otimes\Gg_2\to 0,$$
since $H^2(\Aa\otimes\Gg_1\otimes\Gg_2)=0$ and $H^2(\Aa\otimes\Gg_2(1,0))=0$ we get $H^2(\Aa\otimes\Gg_2)=0$. In a similar way we get $H^2(\Aa\otimes\Gg_1)=0$. \\
So the table become
 \begin{center}\begin{tabular}{|c|c|c|c|c|c|c|c|c|c|c|}
\hline
 $\Oo_F(-1,-1)$ & $\Oo_F(-2,0)$ & $\Oo_F(0,-2)$ & $\Oo_F(0,-1)$ & $\Oo_F(-1,0)$ & $\Oo_F$ \\
 \hline
 \hline
$0$	&	$0$	&	$*$		&	$*$		&	$*$		&	$*$	\\
$0$	&	$a$	&	$0$	&	$0$	&	$*$		&	$*$	\\
$0$	& 	$0$	&	$b$	&	$0$	&	$0$	&	$0$	\\
$0$	& 	$0$	&	$0$	&	$0$	&	$0$	&	$0$	\\
$*$		&	$*$ 	 	&	$0$	&	$0$	& 	$0$	& 	$0$	\\
$*$		&	$*$		&	$*$		&	$*$		&	$0$	& 	$0$ \\
\hline
$\Oo_F(-1,-1)$		& $\Oo_F(-1,0)$	 & $\Oo_F(0,-1)$	 & $\Gg_2$		& $\Gg_1$		&$\Oo_F$\\
\hline
\end{tabular}
\end{center}
where $a=h^2(\Aa(-1,0))$ and $b=h^2(\Aa(0,-1))$. Hence we obtain
$$0\to\Oo_F(-2,0)^{\oplus a}\to \Aa\to\Oo_F(0,-2)^{\oplus b}\to 0.$$
 So, twisting by $\Oo_F(2,2)$ we get the claimed extension.
\end{proof}

\section{Rational normal scrolls}\label{sec4}

 Let $S=S(a_0, \ldots, a_n)$ be a smooth rational normal scroll, the image of $\PP (\Ee)$ via the morphism defined by $\Oo_{\PP (\Ee)}(1)$, where $\Ee \cong \oplus_{i=0}^n \Oo_{\PP^1}(a_i)$ is a vector bundle of rank $n+1$ on $\PP^1$ with $0< a_0 \le \ldots \le a_n$. Letting $\pi : \PP (\Ee) \rightarrow \PP^1$ be the projection, we may denote by $H$ and $F$, the hyperplane section corresponding to $\Oo_{\PP(\Ee)}(1)$ and the fibre corresponding to $\pi^*\Oo_{\PP^1}(1)$, respectively. Then we have $Pic (S)\cong \ZZ\langle H,F\rangle$ and $\omega_S \cong \Oo_S(-(n+1)H+(c-2)F)$, where $c:=\sum_{i=0}^n a_i$ is the degree of $S$. We will simply denote $\Oo_S(aH+bF)$ by $\Oo_S(a+b,a)$, in particular, $\Oo_S(F)=\Oo_S(1,0)$. From now on we fix an ample line bundle on $S$ to be $\Oo_S(H)=\Oo_S(1,1)$.

Recall the dual of the relative Euler exact sequence of $S$:
\begin{equation}\label{eq1}
0\to \Omega_{S|\PP^1}^1 (1,1) \to \Bb:=\oplus_{i=0}^{n}\Oo_S(a_i,0) \to \Oo_S(1,1) \to 0,
\end{equation}
and so we have $\omega_{S|\PP^1} \cong \Oo_S(-(n+1)H+cF)\cong \Oo_S(c-n-1, -n-1)$. The long exact sequence of exterior powers associated to (\ref{eq1}) is

\begin{equation}\label{eqq1}
\begin{split}
0\to &\Oo_S(c-n,-n) \to \wedge^n \Bb(-n+1, -n+1) \stackrel{d_{n-1}}{\to}\\ &\wedge^{n-1}\Bb(-n+2, -n+2) \stackrel{d_{n-2}}{\to} \cdots \stackrel{d_1}{\to} \Bb \to \Oo_S(1,1) \to 0.
\end{split}
\end{equation}
Now (\ref{eqq1}) splits into
\begin{equation}\label{eqqq1}
0\to \Omega_{S|\PP^1}^i (i,i)\to \wedge^i \Bb \to \Omega_{S|\PP^1}^{i-1}(i,i) \to 0
\end{equation}
for each $i=1,\ldots, n$, and we have $\mathrm{Im}(d_i\otimes \Oo_S(i-1,i-1))\cong \Omega_{S|\PP^1}^{i}(i,i)\subset \wedge^i \Bb$.

Now, thanks the above sequences, we introduce suitable full exceptional collections that we will use in the next Theorem (see \cite[Example 4.6]{AHMP}):

\begin{align*}
&\Ee_{2n+1}=\Oo_S(-n,-n)[-n],~ \Ee_{2n}=\Oo_S(-n+1,-n)[-n], \\
&\Ee_{2n-1}=\Oo_S(-n+1,-n+1)[-n+1],~\Ee_{2n-2}=\Oo_S(-n+2,-n+1)[-n+1],\dots , \\
&~ \Ee_3=\Oo_S(-1,-1)[-1], ~ \Ee_2=\Oo_S(0,-1)[-1], \Ee_1=\Oo_S(-1,0)~,~ \Ee_0=\Oo_S.
\end{align*}

and the right dual collection 
\begin{align*}
&\Ff_{2n+1}=\Oo_S(c-3,-1)~,~ \Ff_{2n}=\Oo_S(c-2,-1), \\
&\Ff_{2n-1}=\Omega_{S|\PP^1}^{n-1}(n-3,n-1)~,~ \Ff_{2n-2}=\Omega_{S|\PP^1}^{n-1}(n-2,n-1), \dots ,\\
&\Ff_3=\Omega_{S|\PP^1}^1(-1,1)~,~ \Ff_2=\Omega_{S|\PP^1}^1(0,1)~,~ \Ff_1=\Oo_S(-1,0)~,~ \Ff_0=\Oo_S.
\end{align*}

\begin{theorem}\label{volon}
Let $\Vv$ be an Ulrich vector bundle on $S$ such that $h^i(\Vv(-i,-i-1))=0$ for any $i=1,\dots n-1$.  Then $\Vv$ is primitive and arises from an exact sequence of the form:
\begin{equation}\label{res23}
0\to\Oo_S(0,1)^{\oplus a}\to\Vv\to\Oo_S(c-1,0)^{\oplus b}\to 0.\end{equation}

\end{theorem}

\begin{proof}
We consider the Beilinson type spectral sequence associated to $\Aa:=\Vv(-1,-1)$ and consider the full exceptional collection $\Ee_{\bullet}$ and right dual collection $\Ff_{\bullet}$ above. We construct a Beilinson complex, quasi-isomorphic to $\Aa$, by calculating $H^{i+k_j}(\Aa\otimes \Ee_j)\otimes \Ff_j$ with  $i,j \in \{1, \ldots, 2n+2\}$ to get the following table. Here we use several vanishing in the intermediate cohomology of $\Aa, \Aa(-1,-1),\cdots,  \Aa(-n,-n)$ together with vanishing of cohomology $H^0$ and $H^{n+1}$:

\begin{center}\begin{tabular}{|c|c|c|c|c|c|c|c|}
\hline
$\Ff_{2n+1}$ &$\Ff_{2n}$&$\Ff_{2n-1}$& $\Ff_{2n-2}$& \dots
&$\Ff_2$&$\Ff_1$ &$\Ff_0$\\
\hline
\hline
$0$        &$0$        &$0$         &	 $0$	 	& \dots		 	& $0$		 	& $0$			& $0$	\\
$0$        &$H^{n}$    &$0$		    &	 $0$	& \dots		 	& $0$		 	& $0$			& $0$	\\
$0$        &$H^{n-1}$  &$0$	        &	 $H^n$	 	& \dots		    & $0$		    & $0$			& $0$	\\
$0$        &$H^{n-2}$  &$0$	        &	 $H^{n-1}$	& \dots		 	& $0$		 	& $0$			& $0$	\\
\vdots     &\vdots     &\vdots	    &	 \vdots	    & \vdots		& \vdots		& 	\vdots		& \vdots	\\
$0$        &$0$        &$0$		    &	 $H^1$	 	& \dots	 	    & $H^n$		 	& $0$		& $0$	\\
$0$        &$0$        &$0$		    &	 $0$	 	& \dots	 	    & $H^{n-1}$		 	& $H^n$			& $0$	\\
$0$        &$0$        &$0$		    &	 $0$	 	& \dots	 	    & $H^{n-2}$		 	& $H^{n-1}$		& $0$	\\
\vdots     &\vdots     &\vdots	    &	 \vdots	    & \vdots		& \vdots		& 	\vdots		& \vdots	\\
$0$        &$0$        &$0$		    &	 $0$	 	& \dots		 	& $H^1$		 	& $H^2$			& $0$	\\
$0$        &$0$        &$0$		    &	 $0$	 	& \dots		 	& $0$		 	& $H^1$			& $0$	\\
$0$        &$0$        &$0$		    &	 $0$	 	& \dots		 	& $0$		 	& $0$			& $0$	\\

\hline
\hline
$\Ee_{2n+1}$ &$\Ee_{2n}$ &$\Ee_{2n-1}$& $\Ee_{2n-2}$ & \dots
 &$\Ee_2$& $\Ee_1$ &$\Ee_0$\\
 \hline
\end{tabular}
\end{center}

\noindent By \cite[Example 4.6]{AHMP}, we may conclude that all the entries off the diagonal must be zero and thus we get

\begin{center}\begin{tabular}{|c|c|c|c|c|c|c|c|}
\hline
$\Ff_{2n+1}$ &$\Ff_{2n}$&$\Ff_{2n-1}$& $\Ff_{2n-2}$& \dots
&$\Ff_2$&$\Ff_1$ &$\Ff_0$\\
\hline
\hline

$0$        &$0$        &$0$         &	 $0$	 	& \dots		 	& $0$		 	& $0$			& $0$	\\
$0$        &$H^n$    &$0$		    &	 $0$	& \dots		 	& $0$		 	& $0$			& $0$	\\
$0$        &$0$  &$0$	        &	 $0$	 	& \dots		    & $0$		    & $0$			& $0$	\\
$0$        &$0$  &$0$	        &	 $H^{n-1}$	& \dots		 	& $0$		 	& $0$			& $0$	\\
\vdots     &\vdots     &\vdots	    &	 \vdots	    & \vdots		& \vdots		& 	\vdots		& \vdots	\\
$0$        &$0$        &$0$		    &	 $0$	 	& \dots	 	    & $0$		 	& $0$		& $0$	\\
$0$        &$0$        &$0$		    &	 $0$	 	& \dots	 	    & $0$		 	& $0$			& $0$	\\
$0$        &$0$        &$0$		    &	 $0$	 	& \dots	 	    & $0$		 	& $0$		& $0$	\\
\vdots     &\vdots     &\vdots	    &	 \vdots	    & \vdots		& \vdots		& 	\vdots		& \vdots	\\
$0$        &$0$        &$0$		    &	 $0$	 	& \dots		 	& $H^1$		 	& $0$			& $0$	\\
$0$        &$0$        &$0$		    &	 $0$	 	& \dots		 	& $0$		 	& $H^1$			& $0$	\\
$0$        &$0$        &$0$		    &	 $0$	 	& \dots		 	& $0$		 	& $0$			& $0$	\\

\hline
\hline
$\Ee_{2n+1}$ &$\Ee_{2n}$ &$\Ee_{2n-1}$& $\Ee_{2n-2}$ & \dots
 &$\Ee_2$& $\Ee_1$ &$\Ee_0$\\
 \hline
\end{tabular}
\end{center}
 
Notice that the vanishing hypothesis are:
$$h^1(\Vv(-1,-2))=h^i(\Aa(0,-1))=h^i(\Vv\otimes\Ee_2)=0,$$
$$h^2(\Vv(-2,-3))=h^i(\Aa(-1,-2))=h^i(\Vv\otimes\Ee_4)=0,$$
$$\vdots$$
$$h^{n-1}(\Vv(-n+1,-n))=h^i(\Aa(-n+2,-n+1))=h^i(\Vv\otimes\Ee_{2n-2})=0.$$

So we get the following table:
\begin{center}\begin{tabular}{|c|c|c|c|c|c|c|c|}
\hline
$\Ff_{2n+1}$ &$\Ff_{2n}$&$\Ff_{2n-1}$& $\Ff_{2n-2}$& \dots
&$\Ff_2$&$\Ff_1$ &$\Ff_0$\\
\hline
\hline

$0$        &$0$        &$0$         &	 $0$	 	& \dots		 	& $0$		 	& $0$			& $0$	\\
$0$        &$b$    &$0$		    &	 $0$	& \dots		 	& $0$		 	& $0$			& $0$	\\
$0$        &$0$  &$0$	        &	 $0$	 	& \dots		    & $0$		    & $0$			& $0$	\\
$0$        &$0$  &$0$	        &	 $0$	& \dots		 	& $0$		 	& $0$			& $0$	\\
\vdots     &\vdots     &\vdots	    &	 \vdots	    & \vdots		& \vdots		& 	\vdots		& \vdots	\\
$0$        &$0$        &$0$		    &	 $0$	 	& \dots	 	    & $0$		 	& $0$		& $0$	\\
$0$        &$0$        &$0$		    &	 $0$	 	& \dots	 	    & $0$		 	& $0$			& $0$	\\
$0$        &$0$        &$0$		    &	 $0$	 	& \dots	 	    & $0$		 	& $0$		& $0$	\\
\vdots     &\vdots     &\vdots	    &	 \vdots	    & \vdots		& \vdots		& 	\vdots		& \vdots	\\
$0$        &$0$        &$0$		    &	 $0$	 	& \dots		 	& $0$		 	& $0$			& $0$	\\
$0$        &$0$        &$0$		    &	 $0$	 	& \dots		 	& $0$		 	& $a$			& $0$	\\
$0$        &$0$        &$0$		    &	 $0$	 	& \dots		 	& $0$		 	& $0$			& $0$	\\

\hline
\hline
$\Ee_{2n+1}$ &$\Ee_{2n}$ &$\Ee_{2n-1}$& $\Ee_{2n-2}$ & \dots
 &$\Ee_2$& $\Ee_1$ &$\Ee_0$\\
 \hline
\end{tabular}
\end{center}
where $a:=h^1(\Aa(-1,0))=h^1(\Aa\otimes\Ee_1)$ and $b:=h^n(\Aa(-n+1,-n))=h^n(\Aa\otimes\Ee_{2n})$.

This yields to the desired extension.


\end{proof}

\begin{remark}\label{con}
For a rational normal scroll of dimension $n+1$ we need $n-1$ cohomolgical vanishing conditions in order to characterize primitive Ulrich bundles.\\
In particular for $n=1$ we get that any Ulrich bundle is primitive (see \cite{FM}).\\
For $n=2$ the primitive Ulrich bundles are characterized by just one cohomological condition. In particular if $c=3$, $S=\PP^2\times\PP^1$ there are arbitrary large families of ACM but only a finite number of ACM bundles which are not primitive Ulrich (see \cite{FMS}).
\end{remark}

\section{Open problems}\label{sec5}
So far we have considered projective varieties with Picard number two. Let us consider now the case $\ppp$.
Let $V_1, V_2, V_3$ be three $2$-dimensional vector spas with the coordinates $[x_{1i}], [x_{2j}], [x_{3k}]$ respectively with $i,j,k\in \{1,2\}$. Let $X\cong \mathbb P (V_1) \times \mathbb P (V_2) \times \mathbb P (V_3)$ and then it is embedded into $\mathbb P^7\cong \mathbb P(V)$ by the Segre map where $V=V_1 \otimes V_2 \otimes V_3$.

The intersection ring $A(X)$ is isomorphic to $A(\mathbb P^1) \otimes A(\mathbb P^1) \otimes A(\mathbb P^1)$ and so we have
$$A(X) \cong \mathbb Z[h_1, h_2, h_3]/(h_1^2, h_2^2, h_3^2).$$
We may identify $A^1(X)\cong \mathbb Z^{\oplus 3}$ by $a_1h_1+a_2h_2+a_3h_3 \mapsto (a_1, a_2, a_3)$. Similarly we have $A^2(X) \cong \mathbb Z^{\oplus 3}$ by $k_1e_1+k_2e_2+k_3e_3\mapsto (k_1, k_2, k_3)$ where $e_1=h_2h_3, e_2=h_1h_3, e_3=h_1h_2$ and $A^3(X) \cong \mathbb Z$ by $ch_1h_2h_3 \mapsto c$.
Then $X$ is embedded into $\mathbb P^7$ by the complete linear system $h=h_1+h_2+h_3$ as a subvariety of degree $6$ sin $h^3=6$.\\
We have six Ulrich line bundles namely $\Oo_X(2,1,0)$ up to permutations. Notice that $$Ext^1(\Oo_X(0,1,2),\Oo_X(2,1,0))\cong H^1(\Oo_X(2,0,-2))\cong\CC^3$$ and $$Ext^1(\Oo_X(1,0,2),\Oo_X(2,1,0))=H^1(\Oo_X(1,1,-2))=\CC^4$$ so we have two (up to permutations) families of rank two primitive Ulrich bundles arising from the extensions
\begin{equation}\label{q23}
0\to\Oo_X(2,1,0)\to\Vv\to\Oo_X(0,1,2)\to 0.\end{equation}
and \begin{equation}\label{p23}
0\to\Oo_X(2,1,0)\to\Vv\to\Oo_X(1,0,2)\to 0.\end{equation}

{\bf Question 1:} How many and what cohomological conditions are necessary to characterize primitive Ulrich bundles on $X$ or other varieties?

{\bf Question 2:} The number of cohomological conditions is always the same for each family of primitive Ulrich bundles on $X$ or other varieties?

In \cite{CFM2} it has been proved that the moduli space of rank two Ulrich bundles $\mathcal M(h_1+2h_2+3h_3, 4h_2h_3 + h_1h_3 + 2h_1h_2)$ is a single point, representing the equivalence class of all the strictly semistable bundles with such a $c_1$ from (\ref{q23})
and the moduli space $\mathcal M(h_1+2h_2+3h_3, 2h_2h_3 + 2h_1h_3 + 4h_1h_2)$ is generically smooth and rational of dimension 5: its general point corresponds to a stable bundle and it also contains exactly one point representing the equivalence class of all the strictly semistable bundles with such a $c_1$ from (\ref{p23}).

{\bf Question 3:} Which moduli spaces of Ulrich bundles are made up completely of primitive Ulrich bundles and which only partially on $X$ or other varieties?

So far we have considered the two Del Pezzo threefold of degree $6$, the remaining case is the del Pezzo threefold $Y$ of degree $d = 7$.  Rank two ACM bundles on $Y$ are classified in \cite{CFM3} and it is showed that there are not Ulrich line bundles. So on $Y$  no primitive Ulrich bundle can exist. An interesting well known open problem is the following: which is the lowest rank $\delta$ of an indecomposable Ulrich sheaf on a given projective variety? In the case of smooth hypersurfaces $X \subset \mathbb P^N$ Buchweitz, Greuel and Schreyer conjectured (see \cite{BGS}) that the minimal rank $\delta$  of an indecomposable
Ulrich bundle should be at least $2^{\lfloor\frac{n-2}{2}\rfloor}$. True for $\mathcal Q_n$. So for the cases where such a $\delta$ is known we give the following definition:

\begin{definition}
A vector bundle $E$ over a smooth projective variety $X$ is said $\delta$-primitive Ulrich  bundle if it is an Ulrich bundle which is extension of direct sums of Ulrich rank $\delta$ bundles. So $E$ is a $\delta$-primitive Ulrich bundles if there exist $A=\oplus_{i=1}^s U_i$ and $B=\oplus_{j=1}^z U'_j$,  with $U_i,U'_j$  Ulrich rank $\delta$ bundles, such that $E$ arises from the following exact sequence
$$0\to A\to E\to B\to 0.$$

\end{definition}

\begin{remark}\label{con}
On $\mathcal Q_n$ all the Ulrich bundles are $\delta$-primitive.
\end{remark}

{\bf Question 4:} How many and what cohomological conditions are necessary to characterize $\delta$-primitive Ulrich bundles on a smooth projective variety?


\providecommand{\bysame}{\leavevmode\hbox to3em{\hrulefill}\thinspa}
\providecommand{\MR}{\relax\ifhmode\unskip\spa\fi MR }
\providecommand{\MRhref}[2]{%
  \href{http://www.ams.org/mathscinet-getitem?mr=#1}{#2}
}
\providecommand{\href}[2]{#2}

\end{document}